%% file: main.tex
\documentclass[sigconf, nonacm, natbib=false]{acmart} 
\usepackage[backend=biber,style=numeric]{biblatex}
\input{packages}

\input{cmd}

\theoremstyle{plain}

\newtheorem{remark}{Remark}

\title{\input{a_title}}

\author{Xudong Sun}
\affiliation{\department{Institute of AI for Health}\institution{Helmholtz Munich}\country{Germany}}
\authornote{\href{mailto:ac.xudong.sun@gmail.com}{ac.xudong.sun@gmail.com}}
\author{Rene Corbet}
\affiliation{\department{Department of Mathematics\footnote{}}\institution{KTH Royal Institute of Technology}\country{Sweden}
}
\authornote{equal contribution, \href{maths@rene-corbet.de}{maths@rene-corbet.de}, \ddag: work done at KTH, no longer affilicated}
\author{Carsten Marr}
\affiliation{\department{Institute of AI for Health}\institution{Helmholtz Munich}\country{Germany}}

\begin{document}
\setlist[itemize]{noitemsep, topsep=0pt}
\begin{abstract}
\input{abstract}
\end{abstract}
\maketitle 
\section{Introduction}\label{sec:introduction}
\input{sec_introduction}

\section{Preliminaries}\label{sec:preliminaries}
\input{sec_preliminaries}

\input{sec_methods}

\subsection{Implementation details}\label{sec:experiments}
\input{sec_experiments}

\section{Conclusion}\label{sec:conclusion}
\input{sec_conclusion}

\input{future_work}

\section*{Acknowledgments.} 
We thank Michael Kerber and Michael Lesnick for helpful discussions on some experimental technicalities.

\printbibliography

\section*{Appendix}
\appendix
\input{sec_rl}

\end{document}

%% file: packages.tex
\usepackage[utf8]{inputenc}
\usepackage{graphicx}
\usepackage{mathtools}
\usepackage{amsmath}
\usepackage{braket,mathrsfs,url,epsfig,amsfonts}
\usepackage{enumitem,fancyhdr}
\usepackage{latexsym}
\usepackage{hyperref}
\usepackage{cleveref}
\usepackage{algorithm}
\usepackage{algorithmic}

\usepackage{tikz}
\usepackage{tikz-cd}
\usepackage{todonotes}
\usepackage{multicol} 

%% file: cmd.tex
\def\M {{\mathcal{M}}}
\def\N {{\mathcal{N}}}
\def\X {{\mathcal{X}}}

\def\NN{{\mathbb{N}}}

\def\RR{{\mathbb{R}}}
\def\QQ{{\mathbb{Q}}}

\def\top{\bf~Top}
\def\simp{\bf Simp}
\def\vect{\bf Vect}
\def\param{k}
\def\interdist{d_{\bowtie}}
\def\matchdist{d_{\textnormal{match}}}
\def\path{\pi}
\def\pathset{\Pi}
\def\pathdist{d_{\pathset}}

\def\setallpoints{\mathbb{P}}
\def\setadmissiblepoints{\mathcal{A}}

\def\parax{x_1}
\def\paray{x_2}

\let\svtodo\todo
\renewcommand\todo[1]{\svtodo[inline]{#1}}

%% file: a_title.tex
Path representations in multiparameter persistent homology

%% file: abstract.tex
Multiparameter persistence module can capture more topological differences across data instances compared to using a single parameter, where the well-studied matching distance investigates the distance along a straight line in the multiparameter space that gives the biggest difference. We propose to generalize the straight line to a monotone path filtration and offer software implementations.

%


%% file: sec_introduction.tex
Topological data analysis (TDA) has raised attention in the data science community, with applications in drug discovery~\cite{demir2022todd} and machine learning~\cite{Reininghaus2015,CarriereCuturiOudot2017, moor2020topological}.

Multiparameter persistence module captures more topological differences between data instances compared to single parameter persistence
module~\cite{botnan2022introduction}. However, the existing multiparameter persistence module is based on an approximately optimized line. If we extend the line to a path composed of line segments, it can potentially represent better topological differences between data instances. 



Our contributions are:
\begin{itemize}
    \item To our best knowledge, our work is the first to deal with distances along paths instead of straight lines in multiparameter persistence. 
    \item We also investigate Wasserstein distance besides bottleneck distance along these paths. So far, bottleneck distance has been used only for the matching distance.
    \item We provide software implementation for computing distances between data instances along a given path.
\end{itemize}

%% file: sec_preliminaries.tex
\subsection{One-parameter persistence}
We give a brief introduction in this section to one-parameter~\cite{ghrist2008barcodes} and multiparameter~\cite{botnan2022introduction} persistence in the next. 

The essential ingredient for persistent homology is a~\emph{filtration}: given a totally ordered set $P^1$ (usually $\NN$ or $\RR$), 
a filtration is a functor $\X:P^1\rightarrow\top$ where $\top$ denotes the category of topological spaces 
(or $\X:P^1\rightarrow\simp$ where $\simp$ denotes the category of simplicial complexes) 
in which for each $p\leq q\in P^1$, each morphism $\X_{p,q}$ (arrow pointing from source object $X_p$ to target object $X_q$) is an injective map. 
These functors arise naturally from growing spaces, for instance in the case of a given space that gets built up gradually.   

Homology of an arbitrary dimension of a filtration yields the \emph{persistence module}. 
This is a functor $\M:P^1\rightarrow\vect_{\mathbb{K}}$ where $\vect_{\mathbb{K}}$ denotes the category of vector spaces over a fixed field $\mathbb{K}$ coinciding with the coefficients of homology.
This descriptor of the homological behavior of the filtration is free from thresholds or additional parameters. 
Furthermore, it is known to be completely classified by a complete discrete invariant, expressed as~\emph{barcode} or~\emph{persistence diagram}. 
The latter, denoted by $D$, is a collection of points in $\RR^2$. Each point 
\begin{equation}
x_{bd}=(x_b\in P^1, x_d\in P^1)\in D(\mathcal{M})\label{eq:xbd}\end{equation}
with $x_b, x_d \in P^1$ corresponds to a homological feature, its first coordinate $x_b$ represents the birth value of the feature, and the second coordinate $x_d$ represents its death value. 

Let $\eta(\cdot)$ be a matching between two persistence diagrams. Canonical distances between two persistence modules $\M,\N$ are the \emph{bottleneck distance} $d_B$, defined via
{\scriptsize
\begin{equation}
  d_B(\M,\N)\coloneqq \inf_{\eta:D(\M)\rightarrow D(\N)} \sup_{x_{bd}\in D(\M)} \lVert x_{bd} - \eta(x_{bd}) \rVert_\infty,   \label{eq:bottleneck_dist}
\end{equation}
}
and the \emph{Wasserstein distance} $d_{W,q}$, defined via
{\scriptsize
\begin{equation}
  d_{W,q}(\M,\N)\coloneqq {\left[ \inf_{\eta:D(\M)\rightarrow D(\N)} \sum_{x_{bd}\in D(\M)}\lVert x_{bd} - \eta(x_{bd}) \rVert_\infty^q \right]} ^{1/q}
\end{equation}}
with respect to a parameter $q\in[1,\infty)$. $d_B$ may informally be viewed as $d_{W,\infty}$.
These distances can be computed efficiently~\cite{kerber2016geometry}.

\subsection{Multiparameter persistence}
Persistence modules $\M:P^1\rightarrow\vect_{\mathbb{K}}$ readily generalize to the case
where the indexing category $P^{>1}$ is a partially ordered set (poset for short) and we denote $P^{>1}$ as $P$ for simplicity henceforth. 
$P$ is usually set to a direct sum consisting of summands of $\NN,\QQ_{\geq0},\RR_{\geq0}$ or their poset opposites $\NN^{op},\QQ_{\geq0}^{op},\RR_{\geq0}^{op}$. 
In cases with more than one summand, we call these functors \emph{multiparameter persistence modules}, respectively.

Analogously to the one-parameter case, multiparameter persistence modules naturally arise as homology of \emph{multifiltrations}: 
these are functors $\X:P\rightarrow\top$ (or $\X:P\rightarrow\simp$) in which each morphism $\X_{p,q}$ with $p \prec q \in P$ is an injective map. 
Multifiltrations are useful in data-analytic situations where a single filtration parameter is not sufficient to encode the structure of interest in data.
If $P$ can be written as 2 summands and can not be decomposed into more than 2 summands, we may also call the aforementioned terms \emph{bipersistence modules} and \emph{bifiltrations.}

Examples of interesting multifiltrations include superlevel-Rips multifiltrations~\cite{carlsson2009theory}, the density-sensitive multicover bifiltration~\cite{Chazal2011,Sheehy2012} and its computationally feasible equivalent, the rhomboid bifiltration~\cite{corbet2021computing}. 

When $P=\NN^k$, the classical theory of multigraded modules can be applied to persistence modules~\cite{carlsson2009theory}, and hence, the toolkit from commutative algebra is available~\cite{eisenbud1995commutative,miller2005combinatorial}. It is well-known that the decomposition theory of multiparameter persistence modules is complicated~\cite{gabriel1972unzerlegbare,carlsson2009theory} and hence does not admit a complete discrete invariant like the barcode~\cite{Ghrist2008} or the persistence diagram~\cite{Cohen-SteinerEdelsbrunnerHarer2007} as in the case of a single parameter. Therefore there is the need for various insightful invariants~\cite{lesnick2019computing,harrington2017stratifying,thomasdiss,scolamiero2017multidimensional,CCS21}, and new challenges arise when computing distances between multiparameter persistence modules~\cite{bjerkevik2019computing,kerber2019exact,gafvert2017stable}.
While this theory is intricate already, oftentimes there is the need to generalize the domains to the real numbers. To do this, the use of suitable finiteness conditions~\cite{CK18,Lesnick2015,miller2017data,scolamiero2017multidimensional} is essential. 

\subsection{Distances between multiparameter persistence modules}\label{sec:distances}
We give a brief overview of previously established distances between multiparameter persistence modules below.

\subsubsection{The interleaving distance}
The \emph{interleaving distance}~\cite{chazalinterleaving,lesnick2015theory} is a 
generalization of the bottleneck distance defined in~\Cref{eq:bottleneck_dist} on persistence diagrams, viewed on the level of algebra. 

For $\varepsilon\geq0$, we define $\M(\varepsilon)$ to be the persistence module shifted by the all-$\varepsilon$ vector. 
Formally, ${\M(\varepsilon)}_a\coloneqq {\M}_{a+\varepsilon}$ and ${\M(\varepsilon)}_{a,b}\coloneqq \M_{a+\varepsilon,b+\varepsilon}$. 
If the module is indexed by a lower-bounded set like $\RR_{\geq0}^{\param}$, we add 0-vector spaces in ${\M(\varepsilon)}_{a,b}$ whenever $a<\varepsilon$ or $b<\varepsilon$. 
$\M,\N$ are said to be \emph{$\varepsilon$-interleaved} if there exist morphisms $\M\rightarrow\N(\varepsilon)$ and $\N\rightarrow\M(\varepsilon)$ that commute with the linear transformations in $\M$ and $\N$.

Then the \emph{interleaving distance} is defined to be the infimum over all such $\varepsilon$, i.e.,
\begin{equation}
  \interdist(\M,\N) = \inf_{\epsilon}\left\{\varepsilon\geq 0\mid \text{$\M$ and $\N$ are $\varepsilon$-interleaved}\right\}.
\end{equation}
It is known to be the most discriminative stable distance~\cite{lesnick2015theory} but its computation and even its approximation is an NP-hard problem~\cite{bjerkevik2019computing}.

\subsubsection{The matching distance}\label{sec:matching_dist}

The \emph{matching distance}~\cite{cerri2013betti} is another generalization of the bottleneck distance defined in~\Cref{eq:bottleneck_dist} to multiparameter persistence. It is defined as the weighted supremum of the bottleneck distance along a straight line $\ell$ with positive slope in the parameter space, which we call \emph{slice}, and the set of all slices will be denoted by $\mathcal{L}^+$. Formally,
\begin{equation}
\matchdist(\M,\N) = \sup_{\ell\in\mathcal{L}^+}\left( w_\ell\cdot d_B(\M_\ell,\N_\ell)\right)    
\end{equation}
where $\M_\ell$ and $\N_\ell$ are defined in~\Cref{def:module_l} and $w_{\ell}$ defined in~\Cref{eq:weight_slice_line}.
\begin{definition}[$\M_\ell$]
The restrictions of the corresponding multiparameter persistence module $\M$ to the one-parameter persistence module along the slice $\ell$ is $\M_\ell$.\label{def:module_l}
\end{definition}
If the direction of $\ell$ is expressed by a vector $v=(v_1,\ldots,v_\param)$ with unit length, where $\param$ is the number of parameters, 
\begin{equation}
  w_\ell\coloneqq \min_{(1,\ldots,\param)}{|v_i|} \label{eq:weight_slice_line}
\end{equation}
Consequently, the closer the line is to one of the axes, the more the weight $w_\ell$ penalizes the corresponding bottleneck distance. 

There are efficient computational tools like \emph{box approximations} and the \emph{augmented arrangement} for a fast approximation~\cite{kerber2020efficient} and exact computation~\cite{kerber2019exact} in two parameters. 

Note that the matching distance is stable with respect to the interleaving distance~\cite{landi2014rank}, i.e., for all multiparameter persistence modules $\M,\N$ we have $\matchdist(\M,\N)\leq\interdist(\M,\N)$. The weight defined above ensures this inequality.

\begin{remark}
In bifiltrations like the multicover bifiltration, one parameter is discrete, and the other is continuous. 
Therefore, one might like to rescale the parameter space, which would correspond to an adjustment of the weights.
Hence, the weights $w_\ell$ are theoretically reasonable, 
but might not be the most suitable choice in practice. We resolve this issue in~\Cref{sec:stretching}.\label{remark:weigth_path_short_coming}
\end{remark} 

\subsubsection{Other distances}
Other distances include multiparameter versions of ${L}_p$- and Wasserstein-distances~\cite{bjerkevik2021lp,thomasdiss,bubenik2018algebraic}, distances obtained from noise systems~\cite{scolamiero2017multidimensional} and persistence contours~\cite{gafvert2017stable}, as well as distances between the hierarchical stabilization~\cite{gafvert2017stable} of classical invariants. An example of the latter is stable rank~\cite{gafvert2017stable,CCS21} which may also serve as a feature map for machine learning tasks.

Regarding the construction of feature maps, note that the metric geometries of the spaces of persistence diagrams with bottleneck distance and Wasserstein distance are known to be complicated and rich~\cite{bubenik2020embeddings,mileyko2011probability,turner2014frechet}. 
Hence, there is no hope for easier properties in the case of its generalizations to the multiparameter setting. 
In particular, one-parameter persistence modules are far away from being finitely dimensional vector spaces~\cite{bubenik2020embeddings,carriere2018metric,wagner2021nonembeddability} in the sense of metric geometry.
In order to perform machine learning tasks, several feature maps in infinitely dimensional vector spaces have been constructed, both in the case of one parameter~\cite{Reininghaus2015,KusanoFukumizuHiraoka2016,riihimaki2018generalized} and multiple parameters~\cite{corbet2019kernel,vipond2018multiparameter,carriere2020multiparameter,CCS21}.
The $\mathcal{L}^p$-distance between feature maps may serve as additional distance for multiparameter persistence modules. This could also be compared to the $\mathcal{L}^p$-distance of the Hilbert functions of persistence modules~\cite{Keller2018}.

%% file: sec_methods.tex
\section{Method}\label{sec:methods}
\input{sec_stretch}
\subsection{Path distances}\label{sec:construction}
\input{sec_path_distance_construction}

\subsection{Computing distances between point clouds}
Given a path $\path_{t}=p_0, p_1, \ldots, p_t$ where we use $t$ to indicate the variable length of the path in contrast to fixed $n$ used in~\Cref{sec:construction}.
Our method takes as two point clouds (a.k.a.~two data instances) $\mathcal{D}_1, \mathcal{D}_2$ as input, compute their multifiltrations and then computes their~\emph{path multifiltration}s. See~\Cref{algo:querydistance}. 
The path distance is defined completely analogous by the same stretching arguments as in~\Cref{def:path_bottleneck_dist}. In addition to these arguments, the entry values of the multifiltrations are orthogonally projected to the path of interest. 
We restricted to path persistence modules in~\Cref{def:path_bottleneck_dist} for the sake of simpler exposition and omit further details.

\begin{algorithm}[H] 
\caption{query\_distance~($\path_{t}=p_0, p_1, \ldots, p_t, \mathcal{D}_1, \mathcal{D}_2)$} 
\begin{algorithmic}[1]
  \REQUIRE{two datasets $\mathcal{D}_1, \mathcal{D}_2$, path $\path_{t}=p_0, p_1, \ldots, p_t$.}
  \STATE{calculate multi-filtration for  $\mathcal{D}_1, \mathcal{D}_2$.}
  \STATE{project multi-filtration long the path input $\path_{t}$.}
  \STATE{calculate persistence diagrams for the projected multi-filtration.}
  \STATE{rescale the two persistence diagrams.}
  \STATE{calculate wasserstein or bottleneck.}
  \RETURN{distance between the two persistence diagrams.}
\end{algorithmic}\label{algo:querydistance}
\end{algorithm}

\input{sec_algo_path_sample}

%% file: sec_stretch.tex
\subsection{Stretching}\label{sec:stretching}
To resolve \Cref{remark:weigth_path_short_coming},
we need a less classical description of the matching distance. 
For that, assume that $\M$ is indexed over the real numbers.
Now, denote $\M_{w_\ell}$ to be $\M_\ell$ \emph{stretched} by a factor of $w_\ell$, 
i.e., for all $\parax\in\RR$ we define 
\begin{equation}
  \M_{w_\ell}(\parax)\coloneqq \M_{\ell}(w_\ell\cdot \parax)\label{eq:module_stretch}
\end{equation}
and for bi-parameter case (which can be extended to more than 2 parameters), 
\begin{equation}
\M_{w_\ell}(\parax,\paray)\coloneqq \M_{\ell}(w_\ell\cdot \parax, w_\ell\cdot \paray)
\end{equation}

We show in the following lemma that we can replace the weight in the definition of the matching distance~\Cref{sec:matching_dist} with stretching the persistence module:

\begin{lemma}~\label{lem:weight_translation}
Let $\M,\N$ be multiparameter persistence modules and $\ell\in\mathcal{L}^+$.
Then  $w_\ell\cdot d_B(\M_\ell,\N_\ell)=d_B(\M_{w_\ell},\N_{w_\ell})$.
\end{lemma}

\begin{proof}
Recall that we use the notations $D(\cdot)$ for  persistence diagrams and $\eta(\cdot)$ for matchings between persistence diagrams. Following $x_{bd}$ notation in~\Cref{eq:xbd}, we define $y_{bd}$ similarly.
We get
\begin{align}
&d_B(\M_{w_\ell},\N_{w_\ell}) \nonumber\\
  =& \inf_{\eta:D(\M_{w_\ell})\rightarrow D(\N_{w_\ell})} \sup_{y_{bd}\in D(\M_{w_\ell})} \lVert y_{bd} - \eta(y_{bd}) \rVert_\infty \\
  =& \inf_{\tilde\eta:D(\M_\ell)\rightarrow D(\N_\ell)} \sup_{x_{bd}\in D(\M_{\ell})} \lVert w_\ell\cdot x_{bd} - w_\ell\cdot \tilde\eta(x_{bd}) \rVert_\infty \label{eq:switch_diagram}\\
  =& w_\ell  \inf_{\tilde\eta:D(\M_\ell)\rightarrow D(\N_\ell)} \sup_{x_{bd}\in D(\M_\ell)}  \lVert x_{bd} - \tilde\eta(x_{bd}) \rVert_\infty\\
  =& w_\ell \cdot d_B(\M_\ell,\N_\ell)
\end{align}
\Cref{eq:switch_diagram} stems from the fact that the set of all matchings 
$\eta:D(\M_{w_\ell})\rightarrow D(\N_{w_\ell})$ 
are in a canonical bijection with the set of all matchings 
$\tilde\eta:D(\M_{\ell})\rightarrow D(\N_{\ell})$.
In other words, in~\Cref{eq:switch_diagram}, $y_{bd}=w_{\ell}x_{bd}$ is due to definition of relation between $\mathcal{M}_{l}$ and $\mathcal{M}_{wl}$ in~\Cref{eq:module_stretch}, $\eta(\cdot)$ matches $y_{bd}$ in the $\mathcal{M}_{wl}$ space to its matching point from another data instance. 
After mapping $y_{bd}$'s counterpart $x_{bd}$ in the $\mathcal{M}_l$ space via $\tilde{\eta}$, we have to multiply by $w_l$ to reach $\eta(y_{bd})$. 
Note that the distance $||\cdot||$ in~\Cref{eq:switch_diagram} is still in the $\mathcal{M}_{wl}$ space, but the argument for the $\sup$ operator is in $\mathcal{M}_l$ space. Bijection is between $x_{bd}$ and $y_{bd}$, as well as $\tilde{\eta}(x_{bd})$ and $\eta(y_{bd})$.
\end{proof}

%% file: sec_path_distance_construction.tex
One of the reasons the matching distance is well-studied in multiparameter persistence is the fact that it is efficiently computable. 

With the above preparation in~\Cref{sec:stretching}, we generalize the approach of the matching distance by 
replacing the slices with general paths in positive direction. 
Formally, we define:

\begin{definition}[Path $\pi$ in multi-parameter persistence module]\label{def:path}
Let $P$ be a poset consisting of summands $\RR$ or $\RR^{op}$. A~\emph{path in $P$} is a piecewise linear curve carried by a finite ordered set of points $\pi:(p_0, p_1, p_2 , \ldots ,p_n)$ such that $p_0\prec p_1 \prec \cdots \prec p_n$.
We denote the collection of all paths in $P$ by $\pathset(P)$
and the collection of all paths in $P$ carried by $n$ points by $\pathset_n(P)$.
\end{definition}

We now define a persistence module along a path. For this we use the notation 
$\ell_{p,q}$ for the straight line connecting points $p$ and $q$ in Euclidean space. The notation $w_{\ell_{p,q}}$, which is used in~\Crefrange{eq:path_persistence_module}{eq:path_persistence_module_end}, then inherits from~\Cref{eq:weight_slice_line}.

Furthermore, we stretch the paths, motivated by the well-known stability guarantees in the case of slices, and its translation via the conclusion in~\Cref{lem:weight_translation}. Thus we reach~\Cref{def:path_persistence_module}.
\begin{definition}[Path persistence module $\M_\path$]\label{def:path_persistence_module}
Let $P$ be a poset consisting of summands $\RR$ or $\RR^{op}$. 
Let 
$\M,\N$ 
be multiparameter persistence modules over $P$. 
Let $\path\in\pathset(P)$ carried by $(p_0, p_1, p_2 , \ldots ,p_n)$. 
Define the \emph{path persistence module $\M_\path:\RR_{\geq0}\rightarrow \bf{Vect}_{\mathbb{K}}$ of $\M$ along $\path$} iteratively as 
{\scriptsize
\begin{equation}
  \M_\path (x) \coloneqq \M(p_i+(x-\sum_{j=0}^{i-1} w_{\ell_{p_j,p_{j+1}}} \lVert p_{j+1}-p_j \rVert)\cdot  \frac{p_{i+1}-p_i}{w_{\ell_{p_i,p_{i+1}}}\cdot \lVert p_{i+1}-p_i \rVert})\label{eq:path_persistence_module}
\end{equation}}

if there is an $i\in\{0,1,\ldots,n-1\}$ such that 
{\scriptsize\begin{equation}
\sum_{j=0}^{i-1} w_{\ell_{p_j,p_{j+1}}} \lVert p_{j+1}-p_j \rVert
\leq x 
< \sum_{j=1}^{i} w_{\ell_{p_j,p_{j+1}}} \lVert p_{j+1}-p_j\rVert,
\end{equation}}
and
{\scriptsize
\begin{equation}
\M_\path (x) \coloneqq 
\M(p_n+(x-\sum_{j=0}^{n-1} w_{\ell_{p_j,p_{j-1}}} \lVert p_{j+1}-p_j \rVert)\cdot \frac{p_{n}-p_{n-1}}{w_{\ell_{p_{n-1},p_n}}\cdot \lVert p_{n}-p_{n-1} \rVert})
\end{equation}
}

if~\Cref{eq:path_persistence_module_end} holds. 
\scriptsize{\begin{equation}
  \sum_{j=0}^{n-1} w_{\ell_{p_j,p_{j+1}}} \lVert p_{j+1}-p_j\rVert\leq x\label{eq:path_persistence_module_end}.
\end{equation}}
\end{definition}

\begin{remark}
In~\Cref{def:path_persistence_module}, we use $x$ as the \textbf{accumulated} natural coordinate along the path. Note that since each $p_i$ is multi-dimensional vector, we use here a scalar x to incorporate a parameterization of the path from $p_i$ to $p_{i+1}$, and so on. In original coordiante, this distance has to be multiplied by weight to equal the distance in the natural coordinate.
After the last waypoint, the path extends to a line.~\Cref{eq:path_persistence_module} is essentially the last cocordinate $p_i$ plus the direction times the length traveled along that direction.
\end{remark}

\begin{remark}
Note that by definition, the path persistence module starts at the first of those points carrying the corresponding path. Hence, the path should by default be carried by a collection of points such that the first point would not be greater than the degree of a generator of the multiparameter persistence module.
\end{remark}

\begin{definition}[path distance, bottleneck version]\label{def:path_bottleneck_dist}
The morphisms $\M_\pi(x,y)$ are defined to be those inherited from $\M$ with the corresponding values.

Now, we define
\begin{equation}
d_\pi^B(\M,\N) \coloneqq d_B(\M_\path,\N_\path)  
\end{equation}
and the~\emph{path distance $\pathdist$} via 
\begin{equation}
\pathdist^B(\M,\N) \coloneqq \sup_{\path\in\pathset} d^B_\path(\M,\N).
\end{equation}
If the supremum is achieved by a certain path, we call that path \emph{the best path}. 
\end{definition}

\begin{corollary}
For multiparameter persistence modules $\M,\N$, we get $\pathdist(\M,\N)\geq \matchdist(\M,\N)$.
\end{corollary}

\begin{proof}
The set of all slices in the definition of the matching distance is a subset of $\pathset_2(P)$, which is a subset of $\pathset(P)$. Now,~\Cref{lem:weight_translation} yields the claim.
\end{proof}

\begin{remark}
Note further that persistent homology along all paths in $\pathset(P)$ particularly contains the fibered barcode~\cite{Lesnick2015}.
    The definition is similar to the coherent matching distance~\cite{cerri2019geometrical}. 
    The latter rather transports a matching of a persistence diagram along suitable paths while we simplify the idea to rather project the multifiltration to any path. 
    We have not found any exact computation or approximation of this distance but would be interested in comparing it with the path distance in future work. 
\end{remark}

\subsubsection{Wasserstein alternatives}\label{rem:wasserstein}
Analogously to using bottleneck distance in~\Cref{def:path_bottleneck_dist}, we could also use Wasserstein distance instead. While using bottleneck distance is known to enjoy theoretical guarantees such as stability, Wasserstein distances may have advantages from a data-scientific point of view: 
While bottleneck distance compares only one pair of points in the persistent diagram, Wasserstein distances give weight to all matched points in the persistence diagrams.
Furthermore, the additional Wasserstein parameter gives more choices that can be freely chosen or learned.

Hence, we define the \emph{$q$-Wasserstein path distance} completely analogous to the path distance in~\Cref{def:path_bottleneck_dist} by replacing bottleneck distance with $q$-Wasserstein distance.
\begin{definition}[path distance, Wasserstein version]\label{def:path_bottleneck_dist}
The morphisms $\M_\pi(x,y)$ are defined to be those inherited from $\M$ with the corresponding values.
Now, we define
\begin{equation}
d_\pi^W(\M,\N) \coloneqq d_W(\M_\path,\N_\path)  
\end{equation}
and the~\emph{path distance $\pathdist^W$} via 
\begin{equation}
\pathdist^W(\M,\N) \coloneqq \sup_{\path\in\pathset} d^W_\path(\M,\N).
\end{equation}
If the supremum is achieved by a certain path, we call that path \emph{the best path}. 
\end{definition}

\begin{remark}
Finding the best path according to~\Cref{def:path_bottleneck_dist} (or an approximation thereof) may also serve as a heuristic to detect regions in the parameter space in which two multiparameter persistence modules differ: 
Contrarily to the Hilbert function, we do not take the dimensions of the vector spaces in the parameter space into account but rather the behavior of persistent homology along the path. 
  Choosing bottleneck distance or Wasserstein distance helps to measure slightly different behavior.
\end{remark}

%% file: sec_algo_path_sample.tex
\subsection{Path generation and optimization}
To construct a path $\path_{t}=p_0, p_1, \ldots, p_t$, we first sample $p_0$ from an initialization set $\mathbb{P}_0\subset P$.~e.g.~$\mathbb{P}_0$ can be strips $(0 \le x_i \le \delta_i)$ with $\delta_i$ being strip size and $i$ indexing the filtration parameters $x$. Based on a partially constructed path with end point $p_{t-1}$, we select next feasible point according to the returned feasible sets in~\Cref{algo:findadmissiblepoints} until the path reached maximum length $T$.



\begin{algorithm}[H] 
\caption{Find\_Next\_Step\_Admissible\_Points~${\setadmissiblepoints}({\path}_{t-1}, {\setallpoints}), \delta, n$}
\begin{algorithmic}[1]
  \REQUIRE{max path length $T$;\\
  $\setallpoints\subset P$ (set of all considered points in filtration parameters space);\\
  current constructed path $\path_{t-1}=p_0,p_1, \ldots, p_{t-1}$;\\
Strip size $\delta_{i}$ for the $ith$ coordinate (filtration parameter);\\
Number of steps to look ahead $n_i$ for each filtration parameter;\\
Let $\max_i(\setallpoints)=\max_i\{x_i(p), \forall p \in \setallpoints\}$;}
\IF{t-1=T or $x_i(p_{t-1})+\delta_i>\max_i(\setallpoints)$}
  \RETURN{$\emptyset$ (no addmissible points to append to $\pi_{t-1}$)}
        \ENDIF{}
        \STATE{$\mathcal{A}=\emptyset$}
        \FOR{$p\in \mathbb{P}:x_i(p_{t-1})<x_i(p)<x_i(p_{t-1})+k_i\delta_{i}$ with $k_i=1, \ldots, n_i$, for each filtration parameter indexed via $i$}
        \STATE{$\mathcal{A}=\mathcal{A}\cup \{p\}$}
        \ENDFOR{}
        \RETURN{$\mathcal{A}$}
	\end{algorithmic}\label{algo:findadmissiblepoints}
\end{algorithm}

%
%

In order to find optimized path to maximally distinguish between two point clouds (data instances), one straightforward way can be sampling an ensemble of paths via sampling sequencially from~\Cref{algo:findadmissiblepoints} and take the best path. In addition, one could utiize already exploited terrains using methods like reinforcement learning which we present details in~\Cref{sec:rl}.

%% file: sec_experiments.tex
We implemented the aforementioned algorithms in \textit{CPP} and \textit{Python}. Our implementation takes as input formats bifiltrations and biboundary matrices. The latter is an output of the implementation~\textsc{Rhomboidtiling}\footnote{https://github.com/geoo89/rhomboidtiling} that constructs the multicover bifiltration. Our software uses existing state-of-the-art software in TDA, namely 
\textsc{CGAL-4.9}~\footnote{Computational Geometry Algorithms Library, https://www.cgal.org},
\textsc{Hera}~\footnote{https://github.com/anigmetov/hera},
\textsc{Mpfree}~\footnote{https://bitbucket.org/mkerber/mpfree},
\textsc{Phat}~\footnote{https://github.com/blazs/phat},
\textsc{Rivet}~\footnote{https://github.com/rivetTDA/rivet}, and
\textsc{Rivet-Python}~\footnote{https://github.com/rivetTDA/rivet-python}. 
These modules are connected in our implmentation ~\footnote{\url{https://github.com/smilesun/multi_parameter_persistence_homology_path_learning}} with CPP, python and pipeline codes and scripts.

%% file: sec_conclusion.tex
We gave the first computationally feasible construction of calculating distances for multiparameter persistence modules along paths rather than straightlines. The construction and the implementation rely on state-of-the art software and our own code in CPP.


%% file: future_work.tex
For future work, it will be interesting to investigate the resolution of the proposed method showing discrepancies between two point clouds (data instances) in comparision to straight-line multiparameter persistence modules and single-parameter persistence modules. In addition, this can be extended to compare two distribution of point clouds.

%
%

%% file: sec_rl.tex
\section{Approximate path optimization with reinforcment leanring}\label{sec:rl}
In~\Cref{algo:constructpathrl}, we use reinforcement learning~\cite{sutton1999reinforcement,sun2020reinbo} to tackle the exploration and exploitation trade-off when traversing a path along the persistence homology module, we implemented a similar reinforcement learning algorithm as in~\cite{sun2020reinbo}.

\input{sec_algo}

%% file: sec_algo.tex
\begin{algorithm}[H] 
	\caption{Construct\_PATH\_RL$(\setallpoints, \setadmissiblepoints, \mathcal{D}_1, \mathcal{D}_2), Q^{(init)}(\cdot, \cdot )$} 
	\begin{algorithmic}[1]
		\REQUIRE maximum length $T$ of a path; 
                Q table $Q(\cdot, \cdot )$ 
        \STATE initialize $S_0=p_0, t_0=1$ 
        \FOR{$t$ in $t_0+1:T$}
      \IF {$\setadmissiblepoints(S_{t-1}, \setallpoints) \neq \emptyset$ (admissible set for next points from~\Cref{algo:findadmissiblepoints})} 
            \IF{$\epsilon \sim Uniform(0,1) < 0.9$ }
	            \STATE $a = \arg\max_{p} Q(S_{t-1}, p)$~s.t. ~$p\in \setadmissiblepoints(S_{t-1}, \setallpoints)$ 
            \ELSE 
                \STATE $a \sim Uniform(\setadmissiblepoints(S_{t-1}, \setallpoints))$ (explore random next points from the set of admissible points)
            \ENDIF
            \STATE $p_t=a, S_{t}=S_{t-1}+p_t$~(string concatenation)  
            \STATE $r$ = query\_distance($\pi=S_t=p_0, \ldots, p_t$, $\mathcal{D}_1, \mathcal{D}_2$) (get reward of current decision from~\Cref{algo:querydistance})
            \STATE Q = update\_Q~($S_{t}, S_{t-1}, a, r$)~(see~\cite{sun2020reinbo})
        \ELSE
            \STATE $S_t=S_{t-1}$
            \STATE break
        \ENDIF
    \ENDFOR
	\RETURN{} $\pi=S_t=p_0, \ldots, p_t$
	\end{algorithmic}\label{algo:constructpathrl}
\end{algorithm}

